\documentclass[a4paper,12pt]{article}     % `article' for A4 paper
\usepackage{amsmath,amscd,amssymb}        % AMS-LaTeX with CD and symbols
\usepackage{latexsym}                     % old LaTeX symbols
\usepackage[english, german]{babel}                % German & English (default)

\usepackage{theorem}                 % theorem extensions

\newtheorem{theorem}{Theorem}
\newtheorem{corollary}{Corollary}

\theorembodyfont{\rmfamily}
\newtheorem{definition}{Definition}

\newenvironment{example}
{\smallskip\noindent{\bf Example\/}.}{\smallskip\par}

\newenvironment{proof}{\begin{ProofwCaption}{Proof}}{\end{ProofwCaption}}
\newenvironment{proof*}[1]{\begin{ProofwCaption}{{#1}}}{\end{ProofwCaption}}
\newenvironment{ProofwCaption}[1]%
  {\addvspace\theorempreskipamount \noindent{\it #1.}\rm}%
  {\qed \par \addvspace\theorempostskipamount}
\newcommand{\qedsymbol}{{\rm $\Box$}}
\newcommand{\qed}{\hfill\qedsymbol}
\newcommand{\CC}{{\mathbb C}}

\newcommand{\QQ}{{\mathbb Q}}

\newcommand{\RR}{{\mathbb R}}
\newcommand{\ZZ}{{\mathbb Z}}

\newcommand{\LLL}{{\mathbb L}}

\newcommand{\be}{{\mathbf{e}}}

\newcommand{\calA}{{\cal A}}

\newcommand{\Var}{{\rm{Var}_{\mathbb{C}}}}

\title{Higher order spectra, equivariant Hodge-Deligne polynomials and Macdonald type equations}
\author{Wolfgang Ebeling and Sabir M.~Gusein-Zade
\thanks{Partially supported by DFG (Mercator fellowship, Eb 102/8-1), RFBR-13-01-00755
and NSh-5138.2014.1.
AMS 2010 Math. Subject Classification: 14L30, 55M35, 57R18.
}
}
\date{}

\begin{document}
\selectlanguage{english}

\maketitle

\begin{center}
{\it Dedicated to Gert-Martin Greuel on the occasion of his 70th birthday}
\end{center}

\begin{abstract}
We define notions of higher order spectra of a complex quasi-projective manifold with an action
of a finite group $G$ and with a $G$-equivariant automorphism of finite order,
some of their refinements and give Macdonald type equations for them.
\end{abstract}

%%%%%%%%%%%%%%%%%%%%%%%
\section*{Introduction}
%%%%%%%%%%%%%%%%%%%%%%%
The (Hodge) spectrum was first defined in \cite{St1, Varch} for a germ of a holomorphic function on $(\CC^n,0)$
with an isolated critical point at the origin. It can be also defined for a pair $(V,\varphi)$, where $V$ is a
complex quasi-projective variety and $\varphi$ is an automorphism of $V$ of finite order: \cite{DL}. (The spectrum
of a germ of a holomorphic function is essentially %% just 
the spectrum of its motivic Milnor fibre defined in \cite{DL}.)

Traditionally the spectrum is defined as a finite collection of rational numbers with integer multiplicities
(possibly negative ones) and therefore as an element of the group ring $\ZZ[\QQ]$
of the group $\QQ$ of rational numbers. Let $K_0^{\ZZ}(\Var)$ be
the Grothendieck group of pairs $(V,\varphi)$, where $V$ is a quasi-projective variety and $\varphi$ is an
automorphism of $V$ of finite order (with respect to the disjoint union). The group $K_0^{\ZZ}(\Var)$ is a ring
with the multiplication defined by the cartesian product of varieties and with the automorphism defined by the
diagonal action. The spectrum can be regarded as a group homomorphism from $K_0^{\ZZ}(\Var)$ to the group ring
$\ZZ[\QQ]$. It is not a ring homomorphism. Rational numbers (i.e. elements of the group $\QQ$) are in one-to-one
correspondence with the elements of the group $(\QQ/\ZZ)\times\ZZ$: $r\longleftrightarrow(\{r\},[r])$, where
$\{r\}$ is the fractional part of the rational number $r$ and $[r]$ is its integer part. In this way the spectrum
can be considered as an element of the group ring $\ZZ[(\QQ/\ZZ)\times\ZZ]$. The corresponding map from
$K_0^{\ZZ}(\Var)$ to $\ZZ[(\QQ/\ZZ)\times\ZZ]$ is a ring homomorphism.

Being a group homomorphism from the Grothendieck ring $K_0^{\ZZ}(\Var)$ to $\ZZ[\QQ]$ or to
$\ZZ[(\QQ/\ZZ)\times\ZZ]$, the spectrum can be regarded as a sort of generalized Euler characteristic.
(The spectrum of a pair $(V,\varphi)$ determines the Euler characteristic of $V$ in a natural way.)
The Euler characteristic and some of its generalizations satisfy Macdonald type equations. For the
Euler characteristic $\chi(\cdot)$ itself this means the following. Let $S^nV=V^n/S_n$ be the $n$th
symmetric power of the variety $V$. Then one has (\cite{Mac})
$$
1+\sum_{n=1}^{\infty}\chi(S^nV)t^n=(1-t)^{-\chi(X)}\,.
$$
(Here and below we consider the additive Euler characteristic defined as the alternating sum of the
dimensions of the cohomology groups with compact support.)

A Macdonald type equation for an invariant is an expression for the generating series of the values of
the invariant for the symmetric powers of a space (or for their analogues) as a series not depending
on the space in the power equal to the value of the invariant for the space itself.
For some generalizations of the Euler characteristic (with values in certain rings) the Macdonald
type equations can be formulated in terms of the so-called power structures over the rings (\cite{GLM1})
which can be defined through (pre-)$\lambda$-ring structures on them. The group ring $\ZZ[\calA]$ of
an abelian group $\calA$ has a natural $\lambda$-structure. The spectrum of a pair $(V,\varphi)$
as an element of $\ZZ[(\QQ/\ZZ)\times\ZZ]$ satisfies a Macdonald type equation.

For a complex quasi-projective manifold $V$ with an action of a finite group $G$ and with a $G$-equivariant
automorphism $\varphi$ of finite order, one can define the notion of an orbifold spectrum as an element
of the group ring $\ZZ[\QQ]$: \cite{ET}. This spectrum takes into account not only the logarithms of the eigenvalues
of the action of the transformation $\varphi$ on the cohomology groups, but also the so-called ages
of elements of $G$ at their fixed points (both being rational numbers). Algebraic manipulations with these
two summands are different. The first ones behave as elements of $\QQ/\ZZ$, whereas the second ones~---
as elements of $\QQ$. This explains why the existence of a Macdonald type equation for this spectrum is
not clear. One can consider refinements of the orbifold spectrum as elements of $\ZZ[(Q/Z)\times\QQ]$
or of $\ZZ[(Q/Z)\times\QQ\times\QQ]$. The last notion is equivalent to the notion of the equivariant
orbifold Hodge--Deligne polynomial.

We define higher order spectra of a triple $(V,G,\varphi)$ with a quasi-projective $G$-manifold $V$,
some of their refinements and give Macdonald type equations for them.

%%%%%%%%%%%%%%%%%%%%%%%%%%%%%%%%
\section{$\lambda$-structure on the group ring of an abelian group}\label{lambda}
%%%%%%%%%%%%%%%%%%%%%%%%%%%%%%%%
Let $\calA$ be an abelian group (with the sum as the group operation) and let $\ZZ[\calA]$ be
the group ring of $\calA$. The elements of $\ZZ[\calA]$ are finite sums of the form
$\sum\limits_{a\in\calA}k_a \{a\}$ with $k_a\in\ZZ$. (We use the curly brackets in order to avoid ambiguity
when $\calA$ is a subgroup of the group $\RR$ of real numbers: $\ZZ$ or  $\QQ$.) The ring operations on
$\ZZ[\calA]$ are defined by $\sum k'_a \{a\}+\sum k''_a \{a\}=\sum (k'_a+k''_a)\{a\}$,
$(\sum k'_a \{a\})(\sum k''_a \{a\})=\sum\limits_{a,b\in\calA} (k'_a\cdot k''_b)\{a+b\}$.

Let $R$ be a commutative ring with a unit. %% ???
A $\lambda$-ring structure on $R$ (sometimes called a pre-$\lambda$-ring structure; see, e.g., \cite{Knutson})
is an ``addi\-ti\-ve-to-multi\-pli\-ca\-ti\-ve'' homomorphism $\lambda: R\to 1+T\cdot R[[T]]$ ($a\mapsto\lambda_a(T)$)
such that $\lambda_a(T)=1+aT+\ldots$. This means that $\lambda_{a+b}(T)=\lambda_a(T)\cdot\lambda_b(T)$ for 
$a,b\in R$.

The notion of a $\lambda$-ring structure is closely related to the notion of a power structure defined
in \cite{GLM1}. Sometimes a power structure has its own good description which permits to use it, e.g.,
for obtaining formulae for generating series of some invariants.
A power structure over a ring $R$ is a map $(1+T\cdot R[[T]])\times R\to 1+T\cdot R[[T]]$,
$(A(T),m)\mapsto (A(T))^m$ ($A(t)=1+a_1T+a_2T^2+\ldots$, $a_i\in R$, $m\in R$) possessing all
the basic properties of the exponential function: see \cite{GLM1}. A $\lambda$-structure
on a ring defines a power structure over it. On the other hand there are, in general, many
$\lambda$-structures on a ring corresponding to one power structure over it.

The group ring $\ZZ[\calA]$ of an abelian group $\calA$ can be considered as a $\lambda$-ring.
The $\lambda$-ring structure on $\ZZ[\calA]$ is natural and must be well known. However,
we have not found its description in the literature. Therefore we give here a definition of
a $\lambda$-structure on the ring $\ZZ[\calA]$. (A similar construction was discussed in \cite{GLM4}
for the ring of formal ``power'' series over a semigroup with certain finiteness properties.)

The group ring $\ZZ[\calA]$ can be regarded as the Grothendieck ring of the group semiring
$S[\calA]$ of maps of finite sets to the group $\calA$.
Elements of $S[\calA]$ are the equivalence classes of the pairs $(X,\psi)$ consisting
of a finite set $X$ and a map $\psi:X\to\calA$. (Two pairs $(X_1,\psi_1)$ and $(X_2,\psi_2)$ are equivalent
if there exists a one-to-one map $\xi:X_1\to X_2$ such that $\psi_2\circ\xi=\psi_1$.) The group ring
$\ZZ[\calA]$ is the Grothendieck ring of the semiring $S[\calA]$. Elements of the ring $\ZZ[\calA]$
are the equivalence classes of maps of finite virtual sets (i.e.\ formal differences of sets) to $\calA$.
For a pair $(X,\psi)$ representing an element $a$ of the semiring $S[\calA]$, let its $n$th symmetric power
$S^n(X,\psi)$ be the pair $(S^nX,\psi^{(n)})$ consisting of the $n$th symmetric power $S^nX=X^n/S_n$ of
the set $X$ and of the map $\psi^{(n)}:S^nX\to\calA$ defined by
$\psi^{(n)}(x_1,\ldots, x_n)=\sum\limits_{i=1}^n\psi(x_i)$. One can easily see that the series
$$
\lambda_a(T)=1+[S^1(X,\psi)]T+[S^2(X,\psi)]T^2+[S^3(X,\psi)]T^3+\ldots
$$
defines a $\lambda$-structure on the ring $\ZZ[\calA]$ (or rather a $\lambda$-structure on the semiring
$S[\calA]$ extendable to a $\lambda$-structure on $\ZZ[\calA]$ in a natural way).

The power structure over the ring $\ZZ[\calA]$ corresponding to this $\lambda$-structure can be
described in the following way. Let $A(T)=1+a_1T+a_2T^2+\ldots$, where $a_i=[(X_i,\psi_i)]$, $m=[(M,\psi)]$
with finite {\em sets} $X_i$ and $M$ (thus $a_i$ and $m$ being actually elements of the semiring $S[\calA]$).
Then
$$
\left(A(T)\right)^m= 1+
\sum_{n=1}^\infty
\left(
\sum_{\{n_i\}:\sum in_i=n}
\left[
\left(
((
M^{\sum_i n_i}\setminus\Delta
)
\times\prod_i X_i^{n_i}
)
\left/\prod_i S_{n_i}\right.,
\psi_{\{n_i\}}
\right)
\right]
\right)\cdot T^n\,,
$$
where $\Delta$ is the big diagonal in $M^{\sum_i n_i}$ (consisting of $(\sum n_i)$-tuples of points of $M$
with at least two coinciding ones), the group $\prod_i S_{n_i}$ acts on
$(M^{\sum_i n_i}\setminus\Delta)\times\prod_i X_i^{n_i}$ by permuting simultaneously the factors in
$M^{\sum_i n_i}=\prod_i M^{n_i}$ and in $\prod_i X_i^{n_i}$,
$$
\psi_{\{n_i\}}(\{y_i^j\}, \{x_i^j\})=\sum_i (i\cdot \psi(y_i^j)+\psi_i(x_i^j))\,,
$$
where $y_i^j$ and $x_i^j$, $j=1,\ldots, n_i$, are the $j$th components of the point in $M^{n_i}$
and in $X_i^{n_i}$ respectively
(cf.\ \cite[Equation~(1)]{GLM1}); a similar construction for the Grothendieck ring of quasi-projective
varieties with maps to an abelian manifold was introduced in \cite{MorShen}).

The ring $R[z_1, \ldots, z_n]$ of polynomials in $z_1, \ldots, z_n$ with the coefficients
from a $\lambda$-ring $R$ carries a natural $\lambda$-structure: see, e.g., \cite{Knutson}.
The same holds for the ring $R[z_1^{1/m}, \ldots, z_n^{1/m}]$ of fractional power
polynomials in $z_1, \ldots, z_n$. In terms of the corresponding power structure one can write
$$
(1-T)^{-\sum_{\underline{k}}a_{\underline{k}}\underline{z}^{\underline{k}}}=
\prod_{\underline{k}}\lambda_{a_{\underline{k}}}(\underline{z}^{\underline{k}}T)\,,
$$
where $\underline{z}=(z_1, \ldots, z_n)$, ${\underline{k}}=(k_1, \ldots, k_n)$,
$\underline{z}^{\underline{k}}=z_1^{k_1}\cdot\ldots\cdot z_n^{k_n}$.

The ring $R(G)$ of representations of a group $G$
%% (and the subring $R_1(G)$ generated by the one-dimensional representations)
is regarded as a $\lambda$-ring with the $\lambda$-structure
defined by
$$
\lambda_{[\omega]}(T)=1+[\omega]t+[S^2\omega]T^2+[S^3\omega]T^3+\ldots\,,
$$
where $\omega$ is a representation of $G$, $S^n\omega$ is its $n$th symmetric power.

%%%%%%%%%%%%%%%%%%%%%%%%%%%%%%%%%%%%%%%%%%%%%%%%%%%%%%%%%%%%%%%%%%%
\section{The spectrum and the equivariant Hodge-Deligne polynomial}\label{spectrum}
%%%%%%%%%%%%%%%%%%%%%%%%%%%%%%%%%%%%%%%%%%%%%%%%%%%%%%%%%%%%%%%%%%%
Let $V$ be a complex quasi-projective variety with an automorphism $\varphi$ of finite order.
For a rational $\alpha$, $0\le\alpha<1$, let $H^k_{\alpha}(V)$ be the subspace of $H^k(V)=H^k_c(V;\CC)$
(the cohomology group with compact support) consisting of the eigenvectors of $\varphi^*$ with
the eigenvalue $\be[\alpha]:=\exp{(2\pi\alpha i)}$. The subspace $H^k_{\alpha}(V)$ carries a natural
mixed Hodge structure.

\begin{definition} (see, e.g., \cite{DL})
 The {\em (Hodge) spectrum} ${\rm hsp}(V,\varphi)$ of the pair $(V,\varphi)$ is defined by
 $$
 {\rm hsp}(V,\varphi)=\sum_{k,p,q,\alpha}(-1)^k\dim(H^k_{\alpha}(V))^{p,q}\cdot\{p+\alpha\}\in\ZZ[\QQ]\,.
 $$
\end{definition}

The spectrum ${\rm hsp}(V,\varphi)$ can be identified either with the fractional power polynomial
(Poincar\'e polynomial)
$$
p_{(V,\varphi)}(t)=\sum_{k,p,q,\alpha}(-1)^k\dim(H^k_{\alpha}(V))^{p,q}\cdot t^{p+\alpha}\in\ZZ[t^{1/m}]
$$
or with the equivariant Poincar\'e polynomial
$$
\overline{e}_{(V,\varphi)}(t)=
\sum_{k,p,q,\alpha}(-1)^k\dim(H^k_{\alpha}(V))^{p,q}\omega_{\be[\alpha]}\cdot t^{p}\in R_f(\ZZ)[t]\,,
$$
where $R_f(Z)$ is the ring of finite order representations of the cyclic group $\ZZ$,
$\omega_{\be[\alpha]}$ is the one-dimensional representation of $\ZZ$ with the character equal to
$\be[\alpha]$ at $1$. Both rings $\ZZ[t^{1/m}]$ and $R_f(\ZZ)[t]$ carry natural $\lambda$-structures
and thus power structures. However, the natural power structure on $\ZZ[t^{1/m}]$ is not compatible with
the multiplication: the map 
$$
p_{\bullet}:K_0^{\ZZ}(\Var)\to \ZZ[t^{1/m}]
$$
is not a ring homomorphism. Therefore a natural Macdonald type equation for the spectrum is formulated in
terms of the equivariant Poincar\'e polynomial $\overline{e}_{(V,\varphi)}(t)$. Moreover, a stronger
statement can be formulated in terms of the equivariant Hodge-Deligne polynomial of the pair $(V,\varphi)$.

\begin{definition} (\cite{DimcaL}, see also \cite{Stapledon})
 The {\em equivariant Hodge-Deligne polynomial} of the pair $(V,\varphi)$ is
 $$
 e_{(V,\varphi)}(u,v)=
\sum_{k,p,q,\alpha}(-1)^k\dim(H^k_{\alpha}(V))^{p,q}\omega_{\be[\alpha]}\cdot u^{p}v^q\in R_f(\ZZ)[u,v]\,,
 $$
\end{definition}

One has $\overline{e}_{(V,\varphi)}(t)=e_{(V,\varphi)}(t,1)$.

Let $S^nV$ be the $n$th symmetric power of the variety $V$. The transformation $\varphi:V\to V$
defines a transformation $\varphi^{(n)}:S^nV\to S^nV$ in a natural way.

\begin{theorem}\label{Mac-EquiCheah}
%% \begin{proposition}\label{Mac-EquiCheah}
 One has
 \begin{equation}\label{equiCheah}
  1+e_{(V,\varphi)}(u,v)T+e_{(S^2V,\varphi^{(2)})}(u,v)T^2+e_{(S^3V,\varphi^{(3)})}(u,v)T^3+\ldots=
  (1-T)^{-e_{(V,\varphi)}(u,v)}\,,
 \end{equation}
where the RHS of (\ref{equiCheah}) is understood in terms of the power structure over the ring
$R_f(\ZZ)[u,v]$.
%% \end{proposition}
\end{theorem}

The {\em proof} is essentially contained in \cite{Cheah} where J.~Cheah proved an analogue of (\ref{equiCheah})
for the usual (non-equivariant) Hodge-Deligne polynomial. Theorem~\ref{Mac-EquiCheah} can be deduced from
the arguments of \cite{Cheah} by taking care of different eigenspaces.

Theorem~\ref{Mac-EquiCheah} means that the natural map $e_{\bullet}$ from $K_0^{\ZZ}(\Var)$ to $R_f(\ZZ)[u,v]$
is a $\lambda$-ring homomorphism.

\begin{corollary}\label{Mac-EquiCheah-reduced}
 One has
 \begin{equation}\label{equiCheah-reduced}
  1+\overline{e}_{(V,\varphi)}(t)T+\overline{e}_{(S^2V,\varphi^{(2)})}(t)T^2+
  \overline{e}_{(S^3V,\varphi^{(3)})}(t)T^3+\ldots=
  (1-T)^{-\overline{e}_{(V,\varphi)}(t)}\,,
 \end{equation}
where the RHS of (\ref{equiCheah-reduced}) is understood in terms of the power structure over the ring
$R_f(\ZZ)[t]$.
\end{corollary}

%%%%%%%%%%%%%%%%%%%%%%%%%%%%%%%%%%%%%%%%%%%%%%%%%%%%%%%%%%%%%%%%%%%
\section{The orbifold spectrum and the equivariant orbifold Hodge-Deligne polynomial}\label{orbifold}
%%%%%%%%%%%%%%%%%%%%%%%%%%%%%%%%%%%%%%%%%%%%%%%%%%%%%%%%%%%%%%%%%%%
Let $V$ be a complex quasi-projective manifold of dimension $d$
with an action of a finite group $G$ and with a $G$-equivariant automorphism
$\varphi$ of finite order.
One can say that the notion of the orbifold spectrum of the triple $(V,G,\varphi)$ is inspired by
the notion of the orbifold Hodge-Deligne polynomial: \cite{BD}.

Let $G_*$ be the set of conjugacy classes of elements of $G$.
For an element $g\in G$, let $V^{\langle g\rangle}=\{x\in V:gx=x\}$ be the fixed
point set of $g$, let $C_G(g)=\{h\in G: h^{-1}gh=g\}\subset G$ be the centralizer of $g$.
The group $C_G(g)$ acts on the fixed point set $V^{\langle g\rangle}$. Let $\widehat{\varphi}$
be the transformation of the quotient $V^{\langle g\rangle}/C_G(g)$ induced by $\varphi$.
For a point
$x\in V^{\langle g\rangle}$ the {\em age} of $g$ (or {\em fermion shift number}) is defined in the following way
(\cite[Equation (3.17)]{Zaslow}, \cite[Subsection 2.1]{Ito-Reid}).
The element $g$ acts on the tangent space $T_xV$ as a complex linear operator
of finite order. It can be represented by a diagonal matrix with the diagonal
entries ${\mathbf{e}}[\beta_1]$, \dots, ${\mathbf{e}}[\beta_d]$, where $0\le\beta_i<1$
for $i=1, \ldots, d$ and ${\mathbf{e}}[r]:=\exp{(2\pi i r)}$ for a real number $r$.
The {\em age} of the element $g$ at the point $x$ is defined by 
${\rm age}_x(g)=\sum\limits_{i=1}^d\beta_i\in \QQ_{\ge0}$.
For a rational number $\beta\ge 0$, let $V^{\langle g\rangle}_{\beta}$ be the subspace of the fixed point set
$V^{\langle g\rangle}$ consisting of the point $x$ with ${\rm age}_x(g)=\beta$. (The subspace
$V^{\langle g\rangle}_{\beta}$ of $V^{\langle g\rangle}$ is a union of components of the latter one.)

\begin{definition} (cf.\ \cite{ET})
 The {\em orbifold spectrum} of the triple $(V,G,\varphi)$ is
 $$
 {\rm hsp}^{\rm orb}(V,G,\varphi)=\sum_{[g]\in G_*} \sum_{\beta\in\QQ_{\ge0}}
 {\rm hsp}(V^{\langle g\rangle}_{\beta}/C_G(g),\widehat{\varphi})\cdot\{\beta\}\in\ZZ[\QQ]\,.
 $$
\end{definition}

As above the spectrum ${\rm hsp}^{\rm orb}(V,G,\varphi)$ can be identified with the orbifold Poincar\'e polynomial
$$
p^{\rm orb}_{(V,G,\varphi)}(t)=\sum_{[g]\in G_*} \sum_{\beta\in\QQ_{\ge0}}
 p_{(V^{\langle g\rangle}_{\beta}/C_G(g),\widehat{\varphi})}(t)\cdot t^{\beta}\in\ZZ[t^{1/m}]\,.
$$

It can be regarded as a reduction of the equivariant orbifold Poincar\'e polynomial
$$
\overline{e}^{\rm orb}_{(V,G,\varphi)}(t)=\sum_{[g]\in G_*} \sum_{\beta\in\QQ_{\ge0}}
\overline{e}_{(V^{\langle g\rangle}_{\beta}/C_G(g),\widehat{\varphi})}(t)\cdot t^{\beta}\in R_f(\ZZ)[t^{1/m}]
$$
or of the equivariant orbifold Hodge-Deligne polynomial
$$
 e^{\rm orb}_{(V,\varphi)}(u,v)
 =\sum_{k,p,q,\alpha, [g],\beta}(-1)^k\dim(H^k_{\alpha}(V^{\langle g\rangle}_{\beta}/C_G(g)))^{p,q}
 \omega_{\be[\alpha]}\cdot u^{p}v^q(uv)^{\beta}
$$
(an element of $R_f(\ZZ)[u,v][(uv)^{1/m}]$).

As it was explained above, the presence of (rational) summands of different nature~--- elements of the quotient $\QQ/\ZZ$
and elements of $\QQ$ itself~--- leads to the situation when the existence of a Macdonald type equation for the
orbifold spectrum (and for the orbifold Poincar\'e polynomial) is doubtful. On the other hand there exist Macdonald
type equations for the equivariant orbifold Poincar\'e polynomial and for the equivariant orbifold
Hodge-Deligne polynomial (see Section~\ref{higher}). This inspires the definition of the corresponding versions of spectra.

\begin{definition}
 The {\em orbifold pair spectrum} ${\rm hsp}_2^{\rm orb}(V,G,\varphi)$ of $(V,G,\varphi)$ is
 $$
 \sum_{[g]\in G_*} \sum_{\beta\in\QQ_{\ge0}}
 \sum_{k,p,q,\alpha}(-1)^k\dim(H^k_{\alpha}(V^{\langle g\rangle}_{\beta}/C_G(g),\widehat{\varphi}))^{p,q}
 \{(\alpha, p+\beta)\}\in\ZZ[(\QQ/\ZZ)\times\QQ]\,.
 $$
 The {\em orbifold triple spectrum} ${\rm hsp}_3^{\rm orb}(V,G,\varphi)$ of $(V,G,\varphi)$ is
 $$
 \sum_{[g]\in G_*} \sum_{\beta\in\QQ_{\ge0}}
 \sum_{k,p,q,\alpha}(-1)^k\dim(H^k_{\alpha}(V^{\langle g\rangle}_{\beta}/C_G(g),\widehat{\varphi}))^{p,q}
 \{(\alpha, p+\beta,q+\beta)\}\in\ZZ[(\QQ/\ZZ)\times\QQ\times\QQ]\,.
 $$
\end{definition}

%%%%%%%%%%%%%%%%%%%%%%%%%%%%%%%%%%%%%%%%%%%%%%%%%%%%%%%%%%%%%%%%%%%
\section{Higher order spectrum and equivariant Hodge-Deligne polynomial}\label{higher}
%%%%%%%%%%%%%%%%%%%%%%%%%%%%%%%%%%%%%%%%%%%%%%%%%%%%%%%%%%%%%%%%%%%
The notions of the higher order spectrum of a triple $(V,G,\varphi)$ and of the higher order
equivariant Hodge-Deligne polynomial of it are inspired by the notions of the higher order
Euler characteristic (\cite{AS}, \cite{HH}) and of the corresponding higher order generalized
Euler characteristic (\cite{GLM5}).

Let $(V,G,\varphi)$, $V^{\langle g\rangle}_{\beta}$ and $\widehat{\varphi}$ be as in Section~{\ref{orbifold}}
and let $k\ge 1$.

\begin{definition}
 The {\em spectrum of order} $k$ of the triple $(V,G,\varphi)$ is
 $$
 {\rm hsp}^{(k)}(V,G,\varphi)=\sum_{[g]\in G_*} \sum_{\beta\in\QQ_{\ge0}}
 {\rm hsp}^{(k-1)}(V^{\langle g\rangle}_{\beta}, C_G(g), \varphi)\cdot\{\beta\}\in\ZZ[\QQ]\,.
 $$
 where ${\rm hsp}^{(0)}(V,G,\varphi):={\rm hsp}(V/G,\widehat{\varphi})$.
\end{definition}

The orbifold spectrum is the spectrum of order $1$.

Like above the spectrum of order $k$ can be described by the corresponding order $k$ Poincar\'e polynomial:
$$
p^{(k)}_{(V,G,\varphi)}(t)=\sum_{[g]\in G_*} \sum_{\beta\in\QQ_{\ge0}}
 p^{(k-1)}_{(V^{\langle g\rangle}_{\beta}, C_G(g), \varphi)}(t)\cdot t^{\beta}\in\ZZ[t^{1/m}]\,,
$$
where $p^{(1)}_{\bullet}(t):=p^{\rm orb}_{\bullet}(t)$.

It can be regarded as a reduction of the {\em equivariant order $k$ Poincar\'e polynomial}
$$
\overline{e}^{(k)}_{(V,G,\varphi)}(t)=\sum_{[g]\in G_*} \sum_{\beta\in\QQ_{\ge0}}
\overline{e}^{(k-1)}_{(V^{\langle g\rangle}_{\beta}, C_G(g), \varphi)}(t)\cdot t^{\beta}\in R_f(\ZZ)[t^{1/m}]
$$
or of the {\em equivariant order $k$ Hodge-Deligne polynomial}
 $$
 e^{(k)}_{(V,G,\varphi)}(u,v)=\sum_{[g]\in G_*} \sum_{\beta\in\QQ_{\ge0}}
 e^{(k-1)}_{(V^{\langle g\rangle}_{\beta}, C_G(g), \varphi)}(u,v)(uv)^{\beta}\in R_f(\ZZ)[u,v][(uv)^{1/m}]\,.
 $$

The following definition is an analogue of the definition of the orbifold pair and triple spectra
in Section~\ref{orbifold}.
 
 \begin{definition}
 The {\em pair spectrum of order} $k$ of $(V,G,\varphi)$ is
 $$
 {\rm hsp}_2^{(k)}(V,G,\varphi)=\sum_{[g]\in G_*} \sum_{\beta\in\QQ_{\ge0}}
 {\rm hsp}_2^{(k-1)}(V^{\langle g\rangle}_{\beta},C_G(g),\varphi)
 \{(0, \beta)\}\in\ZZ[(\QQ/\ZZ)\times\QQ]\,.
 $$
 The {\em triple spectrum of order} $k$ of $(V,G,\varphi)$ is
 $$
 {\rm hsp}_3^{(k)}(V,G,\varphi)=\sum_{[g]\in G_*} \sum_{\beta\in\QQ_{\ge0}}
 {\rm hsp}_3^{(k-1)}(V^{\langle g\rangle}_{\beta},C_G(g),\varphi)
 \{(0, \beta, \beta)\}\in\ZZ[(\QQ/\ZZ)\times\QQ\times\QQ]\,.
 $$
\end{definition}
 
The following statement is a Macdonald type equation for the equivariant order $k$ Hodge-Deligne polynomial.
For $n\ge 1$, the cartesian power $V^n$ of the manifold $V$ is endowed with the natural action of the wreath
product $G_n=G\wr S_n=G^n \rtimes S_n$ generated by the componentwise action of the cartesian power $G^n$
and by the natural action of the symmetric group $S_n$ (permuting the factors). Also one has the automorphism
$\varphi^{(n)}$ of $V^n$ induced by $\varphi$. The triple $(V_n, G_n, \varphi^{(n)})$ should be regarded
as an analogue of the symmetric power of the triple $(V,G,\varphi)$.

\begin{example}
 Let $f(z_1,\ldots, z_n)$ be a quasihomogeneous function with the quasi-weights $q_1$, \ldots, $q_n$
 (and with the quasi-degree $1$) and let $G\subset {\rm GL}(n,\CC)$ be a finite group of its symmetries
 ($f(g\overline{z})=f(\overline{z})$ for $g\in G$). The Milnor fibre $M_f=\{f=1\}$ of $f$ is an $(n-1)$-dimensional
 complex manifold with an action of a group $G$ and with a natural finite order automorphism $\varphi$
 (the monodromy transformation or the exponential grading operator):
 $$
 \varphi(z_1,\ldots, z_n)=({\mathbf e}[q_1]z_1,\ldots, {\mathbf e}[q_n]z_n)\,.
 $$
 For $s\ge 1$, let $\CC^{ns}=(\CC^n)^s$ be the affine space with the coordinates $z_i^{(j)}$, $1\le i\le n$,
 $1\le j\le s$. The system of equations $f(z_1^{(j)}, \ldots, z_n^{(j)})=0$, $j=1, \ldots, s$,
 defines a complete intersection
 in $\CC^{ns}$. Its Milnor fibre $M=\{f(z_1^{(j)}, \ldots, z_n^{(j)})=1, \text{ for } j=1, \ldots, s\}$
 is the $s$th cartesian power of the 
 Milnor fibre $M_f$ of $f$ and has a natural action of the wreath product $G_s$.
 The spectrum of a complete intersection singularity is defined by a choice of a monodromy
 transformation. A natural monodromy transformation on $M$ is the $s$th cartesian power $\varphi^{(s)}$ of the 
 monodromy transformation $\varphi$. Thus the triple $(M,G_s,\varphi^{(s)})$ can be regarded
 as an analogue of the $s$th symmetric power of the triple $(M_f,G,\varphi)$.
\end{example}

\begin{theorem}\label{Mac-main}
 Let $V$ be a (smooth) quasi-projective $G$-manifold of dimension $d$ 
 with a $G$ equivariant automorphism $\varphi$ of finite order. One has
\begin{eqnarray}\label{Mac-main-eqn}
&1&+\sum_{n=1}^{\infty}e^{(k)}_{(V^n, G_n,\varphi^{(n)})}(u,v)\cdot T^n\nonumber
\\
&=&\left(\prod\limits_{r_1, \ldots,r_k\geq 1}\left(1-(uv)^{(r_1r_2\cdots r_k)d/2}\cdot T^{r_1r_2\cdots r_k}\right)^{r_2r_3^2\cdots r_k^{k-1}}\right)
^{-e^{(k)}_{(V, G,\varphi)}(u,v)}\,,
\end{eqnarray}
where the RHS is understood in terms of the power structure over the ring $R_f(\ZZ)[u,v][(uv)^{1/m}]$.
\end{theorem}

\begin{proof}
 In \cite{GLM6}, there were defined equivariant generalized higher order Euler characteristics
 of a complex quasi-projective manifold with commuting actions of two finite groups $G_O$ and $G_B$
 as elements of the extension $K_0^{G_B}(\Var)[\LLL^{1/m}]$ of the Grothendieck ring $K_0^{G_B}(\Var)$
 of complex quasi-projective $G_B$-varieties ($\LLL$ is the class of the complex affine line with
 the trivial action) and there were given Macdonald type equations for them: \cite[Theorem~2]{GLM6}.
 One can see that these definitions and the Macdonald type equations can be applied when instead of an
 action of a finite group $G_B$ one has a finite order action of the cyclic group $\ZZ$.
 The equivariant order $k$ Hodge-Deligne polynomial is the image of the equivariant generalized
 Euler characteristic of order $k$ under the map $K_0^{\ZZ}(\Var)[\LLL^{1/m}]\to R_f(\ZZ)[u,v][(uv)^{1/m}]$.
 Since this map is a $\lambda$-ring homomorphism (Theorem~\ref{equiCheah}), the Macdonald type equation
 for the equivariant generalized Euler characteristic of order $k$ implies~(\ref{Mac-main-eqn}).
 \end{proof}
 
 \begin{corollary}
 In the situation described above one has
\begin{eqnarray*}
&1&+\sum_{n=1}^{\infty} {\rm hsp}_2^{(k)}{(V^n, G_n,\varphi^{(n)})}\cdot T^n\\
&=&\left(\prod\limits_{r_1, \ldots,r_k\geq 1}\left(1-\{(0,(r_1r_2\cdots r_k)d/2)\}\cdot T^{r_1r_2\cdots r_k}\right)^{r_2r_3^2\cdots r_k^{k-1}}\right)
^{-{\rm hsp}_2^{(k)}{(V, G,\varphi)}}\,,
\end{eqnarray*}
\begin{eqnarray*}
&1&+\sum_{n=1}^{\infty} {\rm hsp}_3^{(k)}{(V^n, G_n,\varphi^{(n)})}\cdot T^n\\
&=&\left(\prod\limits_{r_1, \ldots,r_k\geq 1}\left(1-\{(0,(r_1r_2\cdots r_k)d/2,(r_1r_2\cdots r_k)d/2)\}\cdot T^{r_1r_2\cdots r_k}\right)^{r_2r_3^2\cdots r_k^{k-1}}\right)
^{-{\rm hsp}_3^{(k)}{(V, G,\varphi)}}\,,
\end{eqnarray*}
where the RHSs are understood in terms of the power structures over the group rings
$\ZZ[(\QQ/\ZZ)\times\QQ]$ and $\ZZ[(\QQ/\ZZ)\times\QQ\times\QQ]$ respectively.
\end{corollary}

\bigskip
\noindent Leibniz Universit\"{a}t Hannover, Institut f\"{u}r Algebraische Geometrie,\\
Postfach 6009, D-30060 Hannover, Germany \\
E-mail: ebeling@math.uni-hannover.de\\

\medskip
\noindent Lomonosov Moscow State University, Faculty of Mechanics and Mathematics,\\
GSP-1, Leninskie Gory 1, Moscow, 119991, Russia\\
E-mail: sabir@mccme.ru

\end{document}